\documentclass[fleqn,11pt]{article}

\usepackage{amssymb, amsmath}
\usepackage{graphics}
\usepackage[dvips]{graphicx}
\usepackage{eepic}
\usepackage{epic}
\usepackage{mathrsfs}
\usepackage{amsmath}
\usepackage{amsfonts}
\allowdisplaybreaks[4]
\usepackage{amsthm}
\usepackage{amscd}
\usepackage[all]{xy}

\usepackage[top=30truemm,bottom=30truemm,left=25truemm,right=25truemm]{geometry}

\newtheoremstyle{break}
  {}%
  {}%
  {\itshape}
  {}%
  {\bfseries}
  {.}%
  {\newline}%
  {}%

\theoremstyle{break}
\newtheorem{thm}{Theorem}[section]
\newtheorem{crl}[thm]{Corollary}

\newtheorem{lmm}[thm]{Lemma}

\newtheorem{prp}[thm]{Proposition}

\newtheorem{dfn}[thm]{Definition}
\newtheorem{exa}[thm]{Example}
\newtheorem{rem}{Remark}

\newcommand{\al}{\alpha}
\newcommand{\be}{\beta}
\newcommand{\ga}{\gamma}
\DeclareMathOperator{\Homo}{H}
\DeclareMathOperator{\Ker}{Ker}
\DeclareMathOperator{\End}{End}
\DeclareMathOperator{\codim}{codim}
\DeclareMathOperator{\rank}{rank}

\usepackage{graphicx}
\usepackage{mathrsfs}

\title{On the rapid decay homology of F.Pham}
\author{Saiei-Jaeyeong Matsubara-Heo\footnote{Graduate School of Mathematical Sciences, The University of Tokyo, 3-8-1 Komaba, Meguro, Tokyo, 153-8914 Japan.\newline e-mail: \texttt{saiei@ms.u-tokyo.ac.jp}}}
\begin{document}
\setlength{\mathindent}{0pt}
\noindent
\maketitle

\begin{abstract}

In \cite{hien}, M. Hien introduced rapid decay homology group $\Homo^{rd}_{*}(U, (\nabla, E))$ associated to an irregular connection $(\nabla, E)$ on a smooth complex affine variety $U$, and showed that it is the dual group of the algebraic de Rham cohomology group $\Homo^*_{dR}(U,(\nabla^{\vee}, E^{\vee}))$. On the other hand, F. Pham has already introduced his version of rapid decay homology when $(\nabla, E)$ is the so-called elementary irregular connection (\cite{Sab}) in \cite{Pham}. In this report, we will state a comparison theorem of these homology groups and give an outline of its proof. This can be regarded as a homological counterpart of the result \cite{Sab} of C. Sabbah. As an application, we construct a basis of some rapid decay homologies associated to a hyperplane arrangement and hypersphere arrangement of Schl\"ofli type. 
\end{abstract}

\section{Motivation}
Gauss hypergeometric function is presumably the most important and well-studied special function. Amongst various properties of Gauss hypergeometric function, the one which enables us to analyse its global behaviour is that it has an integral representation:
$${}_2F_1(\alpha, \beta, \gamma ; x)=\frac{\Gamma(\gamma)}{\Gamma(\gamma -\alpha)\Gamma(\alpha)}\int^1_0t^{\alpha -1}(1-t)^{\gamma -\al -1}(1-xt)^{-\be}dt\;\; (|x|<1).$$

Here, parameters $\al,\be,\ga\in\mathbb{C}$ must satisfy $0<\Re (\al)<\Re (\ga)$ so that the integral is convergent. However, such a restriction can be relaxed by considering the so-called ``regularization'' of paths $[0,1]$ (\cite{AK}), so the essential assumption is only   
$\al, \ga-\al\notin\mathbb{Z}_{\leq 0}$. 

Let us put 

$z_1={}_2F_1(\alpha, \beta, \gamma ; x),$ $z_2=\frac{\Gamma(\gamma)}{\Gamma(\gamma -\alpha)\Gamma(\alpha)}\int^1_0t^{\alpha -1}(1-t)^{\gamma -\al}(1-xt)^{-\be}dt$, $Z={}^t(z_1,z_2).$ 

Then, it can easily be confirmed that $Z$ satisfies the differential equation
\begin{equation}
\frac{d}{dx}Z=
\begin{pmatrix}
\frac{\ga-\al-\be}{x-1} &\frac{\be-\ga}{x-1}\\
\frac{\ga-\al}{x}          &\frac{-\ga}{x}
\end{pmatrix}
Z.
\end{equation}

This suggests that there is another linearly independent solution $\tilde{Z}={}^t(\tilde{z}_1,\tilde{z}_2)$ of (0.1), which is given by the following expression when $\Re(\be)<1$ and $\al, \be, \ga-\al\notin\mathbb{Z}$:

$\tilde{z}_1=\int^{\frac{1}{x}}_0t^{\alpha -1}(1-t)^{\gamma -\al-1}(1-xt)^{-\be}dt,$

$\tilde{z}_2=\int^{\frac{1}{x}}_0t^{\alpha -1}(1-t)^{\gamma -\al}(1-xt)^{-\be}dt.$

Note that the integrands of $\tilde{z}_1$ and $\tilde{z}_2$ are exactly those of $z_1$ and $z_2$. This fact can be recaptured from the view point of the period pairing. In order to formulate the pairing, we prepare several notations. We put

$S=\{ 0,1,\infty, \frac{1}{x}\}$, $\Phi =t^\al(1-t)^{\ga-\al}(1-xt)^{-\be}$,  $U_x=\mathbb{P}^1\setminus S$,

$\nabla =\Phi^{-1}d_t\Phi =d_t+d_tlog\Phi\wedge=d_t+(\al\frac{dt}{t}-(\ga-\al)\frac{dt}{1-t}+\be x\frac{dt}{1-xt})\wedge$.

We equip $U_x$ with Zariski topology and we denote by $\mathcal{O}_{U_x}$ its structure sheaf. Then, $\nabla$ naturally defines a flat algebraic connection $\nabla: \mathcal{O}_{U_x}\rightarrow \Omega_{U_x}^1$. 
We finally put 
$\Homo^1_{dR}(U_x, (\mathcal{O}_{U_x}, \nabla))=\Homo(\Gamma(U_x, \mathcal{O}_{U_x})\overset{\nabla}{\rightarrow}\Gamma(U, \Omega^1_{U_x})\rightarrow 0)$. The following theorem is a simple application of the famous comparison theorem of Deligne and Gr\"othendiek (\cite{Del}).

\begin{thm}\label{thm:DG}
Let $\mathcal{L}$ be the local system of flat sections of $\nabla^{an}$.

There is a perfect pairing given by integration

$$
\begin{array}{cccccc}
\int:&\Homo_1(U_x^{an}, \mathcal{L}^{\vee}) &\times &\Homo^1_{dR}(U_x, (\mathcal{O}_{U_x}, \nabla)) &\longrightarrow &\mathbb{C}\\
 & & \rotatebox{90}{$\in$}& & &\rotatebox{90}{$\in$} \\
 &([\Gamma\otimes\Phi] \hspace{-1cm}&, &\hspace{-2.3cm}[\omega]) &\mapsto &\int_{\Gamma}\Phi\omega
\end{array}
$$
\end{thm}

In the theorem above, $an$ stands for the analytification as usual. Now, since 
$$\Bigl[\frac{dt}{t(1-t)}\Bigr],\; \Bigl[\frac{dt}{t}\Bigr]\in\Homo^1_{dR}(U_x, (\mathcal{O}_{U_x}, \nabla)), \text{ and } [0,1],\;[0,\frac{1}{x}]\in\Homo_1^{lf}(U_x^{an},\mathcal{L}^{\vee})\simeq\Homo_1(U_x^{an},\mathcal{L}^{\vee}),$$ we can verify that $Z$ and $\tilde{Z}$ are periods of this pairing.  Thanks to this description, K. Aomoto and I. M. Gelfand succeeded in introducing a generalization of Gauss hypergeometric function as well as its integral representation and its integration cycles.

On the other hand, it is natural to ask whether we can obtain a similar description of Kummer hypergeometric function. The differential equation for Kummer hypergeometric function can be obtained by the process of confluence. Namely, introducing a small parameter $\epsilon\in\mathbb{C}$, putting $-\epsilon\tilde{x}=x$, $\be=\frac{1}{\epsilon}$, letting $\epsilon\rightarrow +0$, and again putting $\tilde{x}=x$, we obtain a differential equation
\begin{equation}
\frac{d}{dx}Y=
\begin{pmatrix}
-1 & 1\\
\frac{\ga-\al}{x}          &\frac{-\ga}{x}
\end{pmatrix}
Y.
\end{equation}

as well as its solution basis

$Y={}^t(\int^{1}_0t^{\alpha -1}(1-t)^{\gamma -\al-1}e^{-xt}dt,\; \int^{1}_0t^{\alpha -1}(1-t)^{\gamma -\al}e^{-xt}dt)$

$\tilde{Y}={}^t(\int^{\infty}_0t^{\alpha -1}(1-t)^{\gamma -\al-1}e^{-xt}dt,\; \int^{\infty}_0t^{\alpha -1}(1-t)^{\gamma -\al}e^{-xt}dt).$

(Note that the signature of the variable $x$ of the usual Kummer hypergeometric function is different from the one of this paper).

However, we must notice that the direction to the infinity can never be arbitrary in the integration above. This observation was made quite explicit by Bloch-Esnault in the 1-dimensional case and by M. Hien in multidimensional cases.  In order to explain their idea, we put 

$S^\prime=\{ 0,1,\infty\}$, $\Psi =t^\al(1-t)^{\ga-\al}$,  $U=\mathbb{P}^1\setminus S^\prime$,

$\nabla^\prime =e^{xt}\Psi^{-1}d_t\Psi e^{-xt} =d_t+d_tlog(\Psi e^{-xt})\wedge=d_t+(\al\frac{dt}{t}-(\ga-\al)\frac{dt}{1-t}-x dt)\wedge$.

Using similar notations as Theorem \ref{thm:DG}, we can now state a corollary of the result of Bloch-Esnault-Hien for the case of Kummer hyoergeometric function.

\begin{thm}
There is a perfect pairing given by integration
$$
\begin{array}{cccccc}
\int:&\Homo_1^{rd}(U^{an}, \mathcal{L}^{\vee}) &\times &\Homo^1_{dR}(U, (\mathcal{O}_{U}, \nabla^\prime)) &\longrightarrow &\mathbb{C}\\
 & & \rotatebox{90}{$\in$}& & &\rotatebox{90}{$\in$} \\
 &([\Gamma\otimes\Psi] \hspace{-1cm}&, &\hspace{-2.3cm}[\omega]) &\mapsto &\int_{\Gamma}\Psi e^{-xt}\omega
\end{array}
$$

\end{thm}

Here, $\Homo_1^{rd}(U^{an}, \mathcal{L}^{\vee})$ is their invention, the rapid decay homology group, which we shall discuss in this paper. Roughly speaking, the rapid decay homology group is an abelian group of locally finite chains whose direction toward the infinity is the rapid decay direction of the integrand (in this case, it is the rapid decay direction of $e^{-xt}$).

It is now natural to ask how to construct the basis of rapid decay homology group. One way is to employ Morse theory as in \cite{ET}. This method is better suited to the computation of asymptotic expansions at infinity. However, it relies on the existence of critical points, and it is in general untrue that the phase function has as many Morse critical points as the rank of the rapid decay homology group when there appear bounded chambers.

Therefore, in this paper, we focus on a rather older work \cite{Pham} where F. Pham defined his own version of rapid decay homology. The important aspect of Pham's approach is that it enables us to compute the rapid decay homology group in terms of a certain relative homology associated to the phase function. We would like to show that the rapid decay homology group of F.Pham is isomorphic to the one of Bloch-Esnault-Hien and apply it to computations of rapid decay homology groups associated to some elementary irregular connections which comes from hyperplane arrangements.

\section{Rapid decay homology theories}

In this short section, we review the rapid decay homology theory of Bloch-Esnault-Hien and that of F. Pham. Since we are interested in the case of elementary irregular connection, we rewrite the definition of the rapid decay homology group of Bloch-Esnault-Hien (see \cite{hienroucairol}). Throughout this section, we let $U$ denote a complex quasi-projective variety over $\mathbb{C}$.
We take a smooth projective compactification $X$ of $U$ such that $D=X\setminus U$ is a normal crossing divisor.

 We define the oriented blow-up along $D$ following \cite{Sabook}: denoting by $L(D)$ the line bundle associated to $\mathcal{O}_X(D)$, we take a local section $s:X^{an}\rightarrow L(D)$ so that the equation $\{s=0\}=D$ holds locally. This local section naturally induces a section $s:X^{an}\setminus D\rightarrow S^1(D)$ where $S^1(D)=(L(D)\setminus X^{an})/\mathbb{R}_{>0}$. Then, we define the real oriented blow-up $\widetilde{X(D)}$ along $D$ by the formula 
$$\widetilde{X(D)}=(\text{closure of the image of } X^{an}\setminus D\text{ in } S^1(D)).$$
One can easily check that $\widetilde{X(D)}$ does not depend on the choice of $s$, so that we can patch them up to obtain a globally defined topological space $\widetilde{X(D)}\subset S^1(D).$ By construction, there is a projection map $\pi:\widetilde{X(D)}\rightarrow X^{an}$ which is compatible with the projection of a $S^1$ bundle $\pi:S^1(D)\rightarrow X^{an}$. It can also been confirmed that $\widetilde{X(D)}$ has a structure of a differentiable manifold with corners and $\pi$ is a morphism of manifolds with corners in the sense of \cite{Joyce}. Note also that $\widetilde{X(D)}$ can be embedded into a real Euclidian space as a semi-analytic subset. We write $\tilde{X}$ for $\widetilde{X(D)}$ in order to simplify the notation.

Let us describe the morphism $\pi$. Outside the divisor $D$, $\pi$ defines a biholomorphic map $\pi: \tilde{X}\setminus \pi^{-1}(D)\tilde{\rightarrow} X^{an}\setminus D=U^{an}$. On the other hand, locally at $p\in D$, $\pi$ behaves as a polar coordinate, i. e. if $x=(x_1,\cdots ,x_n)$ is a local coordinate such that $D=\{x_1\cdots x_k=0\}$, $\pi$ reads as 
$$
\begin{array}{cccccc}
\pi:&\left([0,\varepsilon )\times S^1\right)^k &\times &\mathbb{C}^{n-k} &\longrightarrow &X^{an}\\
 & & \rotatebox{90}{$\in$}& & &\rotatebox{90}{$\in$} \\
 &((r_i,e^{\sqrt{-1}\theta_i})_{i=1}^k&, & x_{k+1},\cdots ,x_n)&\mapsto &  (r_1e^{\sqrt{-1}\theta_1},\cdots, r_ke^{\sqrt{-1}\theta_k}, x_{k+1}, \cdots, x_n).
\end{array}
$$

Now, suppose that we are given a regular flat algebraic connection $\nabla: E\rightarrow E\otimes\Omega^1_U$ and a non-constant regular function $f$ on $U$. On the oriented blow-up, we can introduce a sheaf $\mathscr{A}^{<D}_{\tilde{X}}$ of holomorphic functions whose asymptotic expansions are 0 along $\tilde{D}=\pi^{-1}(D)$. More precisely, for an open set $\Omega\subset\tilde{X}$, a holomorphic function $u$ on $\Omega\setminus \tilde{D}$ belongs to $\mathscr{A}^{<D}_{\tilde{X}}(\Omega)$ if for any compact subset $K\subset \Omega$ and for any $N=(N_1,\cdots ,N_k)\in\mathbb{Z}_{>0}^k$, there exists $C_{K,N}>0$ such that $u$ satisfies
$$|u(x)|\leq C_{K,N}|x_1|^{N_1}\cdots |x_k|^{N_k}\text{ for all }x\in K\setminus\tilde{D}.$$
 Then, the twisted connection $\nabla_{f}=\nabla -df\wedge$ can be prolonged to a morphism of sheaves $\nabla_f : E\otimes\mathscr{A}^{<D}_{\tilde{X}}\rightarrow E\otimes\mathscr{A}^{<D}_{\tilde{X}}\otimes_{\pi^{-1}\mathcal{O}_{X^{an}}}\pi^{-1}\Omega_{X^{an}}^1(*D)$. We put
$$\mathcal{S}^{<D}=\Ker (\nabla_f : E\otimes\mathscr{A}^{<D}_{\tilde{X}}\rightarrow E\otimes\mathscr{A}^{<D}_{\tilde{X}}\otimes_{\pi^{-1}\mathcal{O}_{X^{an}}}\pi^{-1}\Omega_{X^{an}}^1(*D)).$$

\begin{dfn} 
Let $\mathcal{C}^{-p}_{\tilde{X},\tilde{D}}$ denote the sheaf on $\tilde{X}$ associated to the presheaf
$$V\mapsto S_p(\tilde{X},(\tilde{X}\setminus V)\cup\tilde{D}).$$

We define the sheaf of rapid decay chains $\mathcal{C}^{r.d.,-p}_{\tilde{X}}$ by the formula
$$\mathcal{C}^{r.d.,-p}_{\tilde{X}}=\mathcal{C}^{-p}_{\tilde{X},\tilde{D}}\otimes_\mathbb{C} S^{<D}.$$

Then, if $\mathcal{C}^{r.d.}_{\tilde{X}}$ denotes the complex $\mathcal{C}^{r.d.,-\bullet}_{\tilde{X}}$, we define the $p$-th rapid decay homology group by
$$\Homo_p^{r.d}(U^{an},(E,\nabla_f))=\mathbb{H}^{-p}(\tilde{X}, \mathcal{C}^{r.d.}_{\tilde{X}}).$$
\end{dfn}

The following theorem is a special case of the main result of \cite{hien}.

\begin{thm}

Under the notations above, there is a perfect pairing given by integration
$$
\int:\Homo_p^{rd}(U^{an}, (E^\vee,\nabla_f^\vee)) \times \Homo^p_{dR}(U, (E, \nabla_f)) \longrightarrow \mathbb{C}.
$$
\end{thm}

Hereafter, we discuss the rapid decay homology group of F. Pham.  We use the same notation as above. For any positive real number $c$, we put $S_c^{\pm}=\{ z\in\mathbb{C}|\pm (\Re z-c)>0\}$. By $\Phi$, we denote a family of supports $\Phi =\{ A\overset{closed}{\subset}U^{an}|\forall c>0, A\cap f^{-1}(\bar{S_c^-})\}$. Then, for any given local system $\mathcal{L}$, we define the rapid decay homology of F.Pham by the formula

\begin{dfn}
$$\tilde{\Homo}_*^{r.d.}(U^{an},\mathcal{L})=\Homo_*^{\Phi}(U^{an},\mathcal{L}).$$
\end{dfn}
Note that this homology group is computed as $$\Homo_*^{\Phi}(U^{an},\mathcal{L})=\Homo_*\left(\underset{c>0}{\varprojlim}C_\bullet(U^{an},f^{-1}(S_c^+);\mathcal{L})\right).$$

Let us prove the following simple lemma.

\begin{lmm}\label{lem14}

There exists a finite set $F\subset\mathbb{C}$ such that $f:U^{an}\setminus f^{-1}(F)\rightarrow \mathbb{C}\setminus F$ defines a fiber bundle.

\end{lmm}

\begin{proof}
Take a smooth projective compactification $X$ of $U$ so that the following diagram is commutative
$$
\begin{array}{ccc}
U & \overset{f}{\rightarrow} & \mathbb{A}^1\\
\rotatebox{90}{$\hookleftarrow$} & &\rotatebox{90}{$\hookleftarrow$}\\
X & \overset{f}{\rightarrow} &\mathbb{P}^1. 
\end{array}
$$

Here, $f:X\rightarrow\mathbb{P}^1$ is an extension of $f$. We equip $X$ with a Whitney stratification so that $U$ and $X\setminus U$ are unions of strata. Then, the lemma is an easy consequence of the first isotopy lemma of Thom-Mather (\cite{V}).
\end{proof}

Thanks to this lemma, when $c^\prime >c>0$ are large enough, the morphism of pairs $(U^{an},f^{-1}(S_{c^\prime}^+))\rightarrow (U^{an},f^{-1}(S_c^+))$ induces a homotopy equivalence so that the induced chain map $C_\bullet(U^{an},f^{-1}(S_{c^\prime}^+);\mathcal{L})\rightarrow C_\bullet(U^{an},f^{-1}(S_c^+);\mathcal{L})$ is an isomorphism. Since $S_c^+$ is contractible, we obtain the following long exact sequence:

$$\cdots\rightarrow\Homo_k(f^{-1}(t);\mathcal{L})\rightarrow\Homo_k(U^{an};\mathcal{L})\rightarrow\tilde{\Homo}_k^{r.d.}(U^{an};\mathcal{L})\rightarrow\cdots\;\; \text{(exact)}$$

where $t\in S_c^+$ for $c>0$ large enough.

\begin{exa}
Let us consider the simplest case which was fully investigated in \cite{Pham}. Let $U=\mathbb{A}^n$ and $\mathcal{L}=\mathbb{C}$. Then, by the exact sequence above, we obtain
$$\tilde{\Homo}^{r.d.}_*(\mathbb{C}^n,\mathbb{C})\simeq\tilde{\Homo}_*(f^{-1}(t)^{an},\mathbb{C})$$

where $t\in\mathbb{C}$ is a generic point, and $\sim$ on the right hand side stands for the reduced homology.
\end{exa}

\section{Comparison theorem}

In this section, we are going to prove that the rapid decay homology of F.Pham is isomorphic to that of Bloch-Esnault-Hien when the connection is the so-called elementary irregular connection.


We fix a regular flat algebraic connection $(E,\nabla)$ and a non-constant regular function $f$ on $U$. As before, we suppose that $f$ is extended to a morphism $f:X\rightarrow \mathbb{P}^1$. Consider the lift of $f$ to the oriented blow-up
$$
\begin{array}{ccc}
\tilde{X} & \overset{\tilde{f}}{\rightarrow} & \widetilde{\mathbb{P}^1}\\
\rotatebox{90}{$\leftarrow$} & &\rotatebox{90}{$\leftarrow$}\\
X & \overset{f}{\rightarrow} &\mathbb{P}^1 
\end{array}
$$

where $\widetilde{\mathbb{P}^1}$ is the oriented blow-up of $\mathbb{P}^1$ along $\infty$. Note that for each $\theta\in\mathbb{R}$, the closure of the ray $[0,\infty)e^{\sqrt{-1}\theta}$ in $\widetilde{\mathbb{P}^1}$, and $\widetilde{\mathbb{P}^1}\setminus\mathbb{C}$ has only one intersection point which we will denote by $e^{\sqrt{-1}\theta}\infty$. Then, we have a decomposition $\widetilde{\mathbb{P}^1}=\mathbb{C}\cup S^1\infty$ of the oriented blow-up $\widetilde{\mathbb{P}^1}$. Let us put $\widetilde{D^{r.d.}}=\tilde{f}^{-1}\left(\{e^{\sqrt{-1}\theta}\infty | -\frac{\pi}{2}<\theta <\frac{\pi}{2}\}\right)$, which corresponds to the rapid decay directions of $e^{-f}$. However, since the dimension of $X$ is larger than 1 in general, we need to subtract irrelevant divisors as follows. We first decompose the divisor $D$ as $D=f^{-1}(\infty)\cup D^\prime$. Then, we put $\widetilde{D^\prime}^{r.d.}=\widetilde{D}^{r.d.}\setminus\pi^{-1}(D^\prime)$.

If $U^{an}\overset{i}{\hookrightarrow}U^{an}\cup \widetilde{D^\prime}^{r.d.}\overset{j}{\hookrightarrow}\tilde{X}$ denote the sequence of inclusions, we have the following

\begin{prp}

$$\mathcal{S}^{<D}=j_!i_*\mathcal{L},$$
 where $\mathcal{L}$ is the local system of flat sections of $\nabla^{an}$.
\end{prp}

The proof of this proposition is essentially same as the argument in \cite{hien} pp11-12. We omit the proof.

Since we have an inclusion $(U^{an}\cup\widetilde{D^\prime}^{r.d.},\widetilde{D^\prime}^{r.d.})\subset (\tilde{X},\widetilde{D})$, and $j:U^{an}\cup\widetilde{D^\prime}^{r.d}\rightarrow\tilde{X}$ is an open embedding, we have a morphism 
$$
\mathcal{C}^{-\bullet}_{U^{an}\cup \widetilde{D^\prime}^{r.d.},\widetilde{D^\prime}^{r.d.}}\rightarrow j^{-1}\mathcal{C}^{-\bullet}_{\tilde{X},\widetilde{D}}=j^{!}\mathcal{C}^{-\bullet}_{\tilde{X},\widetilde{D}}
$$

which is induced from the morphism of presheaves 

$$
S(U^{an}\cup \widetilde{D^\prime}^{r.d.},(U^{an}\cup \widetilde{D^\prime}^{r.d.}\setminus V)\cup\widetilde{D^\prime})\rightarrow S(\tilde{X},(\tilde{X}\setminus V)\cup \tilde{D}).
$$

Rigorously speaking, for two presheaves $\mathscr{F}$ and $\mathscr{G}$, we consider a morphism

$$\mathscr{F}\rightarrow j^{-1}\mathscr{G}.$$

This induces a morphism

$$
\mathscr{F}^\dagger\rightarrow(j^{-1}\mathscr{G})^\dagger.
$$

But since we have a canonical morphism $\mathscr{G}\rightarrow\mathscr{G^\dagger}$, we have a morphism of presheaves $j^{-1}\mathscr{G}\rightarrow j^{-1}(\mathscr{G}^\dagger)$ and this induces a morphism of sheaves by the universality of sheafification $(j^{-1}\mathscr{G})^\dagger\rightarrow j^{-1}(\mathscr{G}^\dagger)$. We finally obtain $\mathscr{F}^\dagger\rightarrow j^{-1}(\mathscr{G})^\dagger$.

Now, applying the functor $j_!$ yields to a morphism
$$
j_!\mathcal{C}^{-\bullet}_{U^{an}\cup \widetilde{D^\prime}^{r.d.},\widetilde{D^\prime}^{r.d.}}\rightarrow j_!j^{!}\mathcal{C}^{-\bullet}_{\tilde{X},\widetilde{D}}\rightarrow\mathcal{C}^{-\bullet}_{\tilde{X},\widetilde{D}}.
$$

In the end, we obtain a morphism

$$
j_!\left( \mathcal{C}^{-\bullet}_{U^{an}\cup \widetilde{D^\prime}^{r.d.},\widetilde{D^\prime}^{r.d.}}\otimes i_*\mathcal{L}^\vee\right)\rightarrow \mathcal{C}^{r.d.,-\bullet}_{\tilde{X}}.
$$

In view of the fact that $U^{an}\cup \widetilde{D^\prime}^{r.d.}$ is a manifold with boundary and $j$ is an open embedding, this morphism can be confirmed to be an isomorphism.

\begin{prp}
$\Homo^{r.d.}_p(U^{an},(E ,\nabla))\simeq \Homo_p(U^{an}\cup \widetilde{D^\prime}^{r.d.}, \widetilde{D^\prime}^{r.d.}, i_* \mathcal{L}).$
\end{prp}

\begin{proof}
Denoting by $\Gamma_{\tilde{X}}$ the global section, we obtain
$$
\begin{array}{ccc}
\mathbb{R}\Gamma_{\tilde{X}}\left(j_!\left( \mathcal{C}^{-\bullet}_{U^{an}\cup \widetilde{D^\prime}^{r.d.},\widetilde{D^\prime}^{r.d.}}\otimes i_*\mathcal{L}\right)\right) &\simeq &\hspace{-6mm} \mathbb{R}\Gamma_{\tilde{X}}\circ\mathbb{R}j_!\left(\left( \mathcal{C}^{-\bullet}_{U^{an}\cup \widetilde{D^\prime}^{r.d.},\widetilde{D^\prime}^{r.d.}}\otimes i_*\mathcal{L}\right)\right) \\
 & \simeq & \mathbb{R}\Gamma_c\left(U^{an}\cup\widetilde{D^\prime}^{r.d.},  \mathcal{C}^{-\bullet}_{U^{an}\cup \widetilde{D^\prime}^{r.d.},\widetilde{D^\prime}^{r.d.}}\otimes i_*\mathcal{L}\right)\\
 & \simeq& \hspace{-22mm}C_{-\bullet}(U^{an}\cup\widetilde{D^\prime}^{r.d.},\widetilde{D^\prime}^{r.d.}; i_*\mathcal{L}).
\end{array}
$$
The last isomorphism follows from the fact that the homology  complex is homotopically fine. See \cite{Bredon}.
\end{proof}

Now, we would like to prepare a lemma necessary for the proof of the comparison theorem. It concerns a construction of homotopy which does not preserve the stratification.



\begin{lmm}
$$C_{-\bullet}(U^{an}\cup\widetilde{D^\prime}^{r.d.},\widetilde{D^\prime}^{r.d.}; i_*\mathcal{L})\overset{quasi. isom.}{\longrightarrow} C_{-\bullet}(U^{an}\cup\widetilde{D}^{r.d.},\widetilde{D}^{r.d.}; i_*\mathcal{L}).$$
\end{lmm}

\begin{proof}
Let us remember a basic notations of manifold with corners (\cite{Joyce}). Let $M$ be a manifold with corners.  Let $x=(x^\prime ,x^{\prime\prime})=(x_1,\cdots ,x_l,x_{l+1},\cdots ,x_m)$ be a coordinate at $q\in M$ so that it corresponds to the open set $[0,\varepsilon)^{l}\times(-\varepsilon ,\varepsilon)^{m-l}.$ We may assume $x(q)=0.$ Then the subset 
$$IS_q^+(M)=\sum_{i=1}^{l}\mathbb{R}_{>0}\frac{\partial}{\partial x_i}|_q+\sum_{i=l+1}^{m}\mathbb{R}\frac{\partial}{\partial x_i}|_q$$ of $T_qM$ does not depend on the choice of coordinate.

Consider a canonical morphism of manifold with boundaries $p:\widetilde{X}\rightarrow\widetilde{X(D^\prime)}$ obtained by collapsing $\pi^{-1}(D^\infty)$ to $D^{\infty}$. We then  put 
$$N^+_{\tilde{D^\prime},q}=p_*^{-1}\left(IS^+_{q}(\widetilde{X(D^\prime)})\right).$$
We would like to construct a vector field $\eta$ on $\tilde{X}$ so that
 $$\eta\tilde{f}=0\text{ and }\eta_q\in N^+_{\widetilde{D^\prime}, q}.$$
At each point $q\in D^\infty\cap D^\prime$, we take a coordinate system $z$ of $X$ so that 
\\
$D^\infty=\{ z_1\cdots z_s=0\}$, $D^\prime=\{ z_{s+1}\cdots z_t=0\}$, and $\frac{1}{f}=z_1^{m_1}\cdots z_s^{m_s}$. Now we put locally around $q$, 

$$\eta^{(q)}=\frac{\partial}{\partial r_{s+1}}+\cdots +\frac{\partial}{\partial r_{t}}.$$  

This satisfies the two desired conditions locally. By means of partition of unity, we get the desired $\eta$. Furthermore, by this construction, there is a vector field $\tilde{\eta}$ on $\widetilde{X(D^\prime)}$ such that $\eta_q\in IS^+_q(\widetilde{X(D^\prime)})$ and $p_*\eta=\tilde{\eta}$.

Since $\tilde{X}$ is compact, the flow of $\eta$ is complete and since $p_*\eta=\tilde{\eta}$, we have $p\Phi_t^\eta =\Phi_t^{\tilde{\eta}}$. Now, $\Phi^{\tilde{\eta}}_t(\widetilde{X(D^\prime)})\cap\widetilde{D^\prime}=\phi$ implies
$$\Phi^\eta_t(\tilde{X})\cap\widetilde{D^\prime}=\phi$$ for any $t>0.$ It can be confirmed that $\Phi^\eta_t(U^{an})\subset U^{an}$ since $\Phi^{\tilde{\eta}}_t(U^{an})\subset U^{an}$ and that $\Phi_t^{\eta}(\widetilde{D}^{r.d.})\subset \widetilde{D^\prime}^{r.d.}$ since $\eta \tilde{f}=0.$ Summing up, we have a homotopy equivalence 

$$(U^{an}\cup\widetilde{D^\prime}^{r.d.},\widetilde{D^\prime}^{r.d.})\rightarrow (U^{an}\cup\widetilde{D}^{r.d.},\widetilde{D}^{r.d.}).$$

\end{proof}







Now we are able to prove the following comparison theorem.

\begin{thm}\label{comparison}
$$\tilde{\Homo}^{r.d.}(U^{an}, \mathcal{L})\simeq\Homo^{r.d.}(U^{an},(E,\nabla_f))$$
\end{thm}

\begin{proof}
We put $Y_\varepsilon =f^{-1}(A_\varepsilon)\cup \tilde{f}^{-1}(e^{\sqrt{-1}[-\frac{\pi}{2}+\varepsilon ,\frac{\pi}{2} -\varepsilon]}\cdot \infty)$ where $A_\varepsilon=\mathbb{C}\setminus\{ R+\frac{1}{\varepsilon}+re^{\sqrt{-1}\theta}|0<r, \frac{\pi}{2}-\varepsilon <\theta<\frac{\pi}{2}+\varepsilon\text{ or }\frac{3}{2}\pi-\varepsilon <\theta<\frac{3}{2}\pi+\varepsilon\}.$
We also put $\partial Y_\varepsilon =\tilde{f}^{-1}(e^{\sqrt{-1}[-\frac{\pi}{2}+\varepsilon ,\frac{\pi}{2} -\varepsilon]}\cdot \infty).$

Since $\left(U^{an}\cup\widetilde{D}^{r.d.},\widetilde{D}^{r.d.}\right)=\bigcup_{\varepsilon >0}\left(\text{Int}Y_\varepsilon ,\text{Int}\partial Y_\varepsilon\right)$, we have that
$$\Homo_*(U^{an}\cup\widetilde{D}^{r.d.},\widetilde{D}^{r.d.}; i_*\mathcal{L})\simeq\underset{\varepsilon>0}{\varinjlim}\Homo_*(Y_\varepsilon ,\partial Y_\varepsilon ;\mathcal{L}).$$

Here, since $\tilde{f}:\tilde{X}\rightarrow\tilde{\mathbb{P}^{1}}$ is a proper map, $\{f^{-1}(\{\Re (z)>R^\prime\}\cap A_\varepsilon)\cup\partial Y_\varepsilon\}_{R^\prime>0}$ is a fundamental system of neighbourhoods of $\partial Y_\varepsilon$ in $Y_\varepsilon$. 

Now, since $f^{-1}(\{\Re(z)>R^\prime\}\cap A_\varepsilon)\cup\partial Y_\varepsilon$ and $\partial Y_\varepsilon$ are semi-analytic sets, by the famous result of Lojasiewicz(\cite{Loj}), we obtain that there is a triangulation of  $f^{-1}(\{\Re(z)>R^\prime\}\cap A_\varepsilon)\cup\partial Y_\varepsilon$ such that $\partial Y_\varepsilon$ is a subcomplex of it.

Now, by means of derived neighbourhood, one can show that there is a deformation retract neighbourhood $N$ of $\partial Y_\varepsilon$ (see \cite{hud} lemma2.10). 

Therefore, we have
$$
\begin{array}{ccc}
\Homo_*(Y_\varepsilon ,\partial Y_\varepsilon)&\simeq&\Homo_*\left(\underset{N:\text{ deformation retract nbd of }\partial Y_\varepsilon}{\varprojlim}C_\bullet(Y_\varepsilon ,N)\right)\\
 &\overset{cofinality}{\simeq}&\Homo_*\left(\underset{N:\text{nbd of }\partial Y_\varepsilon}{\varprojlim}C_\bullet(Y_\varepsilon ,N)\right)\\
 &\overset{cofinality}{\simeq}&\Homo_*\left(\underset{R^\prime >0}{\varprojlim}C_\bullet(Y_\varepsilon ,f^{-1}(\{\Re (z)>R^\prime\}\cap A_\varepsilon)\cup\partial Y_\varepsilon)\right)\\
 &\overset{excision}{\simeq}&\Homo_*\left(f^{-1}(A_\varepsilon),f^{-1}(A_\varepsilon\cap\{\Re(z)>R^\prime\})\right).
\end{array}
$$

Note that quasi-isomorphism is preserved under projective limit. On the other hand, the last homology is isomorphic to 
$$\Homo_*\left(U^{an}, f^{-1}\{\Re(z)>R^\prime\}\right)$$
for $R^\prime>0$ large enough.

In view of lemma, we obtain
$$
\begin{array}{ccc}
\Homo_*^{r.d.}(U^{an},(E,\nabla_f))&=&\Homo_*(U^{an}\cup\widetilde{D^\prime}^{r.d},\widetilde{D^\prime}^{r.d};i_*\mathcal{L})\\
 &\simeq&\Homo_*(U^{an}\cup\widetilde{D}^{r.d},\widetilde{D}^{r.d};\mathcal{L})\\
&\simeq&\underset{\varepsilon>0}{\varinjlim}\Homo_*(Y_\varepsilon,\partial Y_\varepsilon ;\mathcal{L})\\
&\simeq&\Homo_*(U^{an},f^{-1}(\{\Re(z)>R^\prime \}))\\
 &\simeq&\tilde{\Homo}_*^{r.d.}(U^{an},\mathcal{L}).
\end{array}
$$ 
\end{proof}

\section{Rapid decay homology groups associated to hyperplane arrangements}

We are going to apply our comparison theorem to elementary irregular connections associated to hyperplane arrangements. First, we remember a result of regular cases.

Let $l_j(t)$ be a real linear polynomial $l_j(t)=a_{0j}+a_{1j}t_1+\cdots +a_{nj}t_n\;\\
(a_{kj}\in\mathbb{R},\; \displaystyle\prod_{k=1}^na_{kj}\neq 0,\;  j=1,\cdots ,N).$ We put

$$A_j=\{t\in\mathbb{A}^n|l_j(t)=0\},\; X=\mathbb{A}^n\setminus\cup_{j=1}^NA_j,\; D=\cup_{j=1}^NA_j.$$

Let further $V$ denote a finite dimensional complex vector space and $P_j$ be an element of $\End(V)$. We introduce a trivial bundle $E=X\times V$ on $X$ and a connection 
$$\nabla=d+\sum_{j=1}^NP_jd\log l_j(t)\wedge :E\rightarrow\Omega_E^1(\log D).$$

By considering residues, we can confirm that $\nabla^{2}=0$ if and only if for any maximal subfamily $\{ A_{j_\nu}\}_{1\leq\nu\leq q}$ such that

$\codim_{\mathbb{C}}(A_{j_1}\cap\cdots\cap A_{j_q})=2,\; 1\leq\forall\nu\leq q,\; [P_{j_\nu},P_{j_1}+\cdots +P_{j_q}]=0.$

Here, $\{ A_{j_\nu}\}_{1\leq\nu\leq q}$ is called a maximal subfamily if 
$$\{ A_j|A_{j_1}\cap\cdots\cap A_{j_q}\subset A_j\}=\{ A_{j_\nu}\}_{1\leq\nu\leq q}.$$

We compactify $\mathbb{A}^{n}$ to $\mathbb{P}^n$ and denote by $A_{N+1}$ the hypersurface at infinity $A_{N+1}=\mathbb{P}^n\setminus\mathbb{A}^n$. By abuse of notation, we denote the closure of $A_j$ in $\mathbb{P}^n$ by the same notation. We also put 
$$P_{N+1}=-\sum_{j=1}^NP_j.$$
As in \cite{Kohno}, we define the following notion.

\begin{dfn} 
We say $(E, \nabla)$ is generic if
$$
\begin{array}{cccc}
 &(1)& &\hspace{-7.6cm}\text{eigenvalues of }P_j\text{ are not integers.}\\
 &(2)& &\text{For any maximal subfamily }\{ A_{j_\nu}\}_{1\leq\nu\leq q}\text{ such that }\codim_{\mathbb{C}}[A_{j_1}\cap\cdots\cap A_{j_q}]=r<q,\\
 &   &  &\hspace{-5.5cm}\text{ any eigenvalue of }P_{j_1}+\cdots P_{j_q}\text{ is not an integer.}
\end{array}
$$
\end{dfn}

The following result has essentially been proved by various authors in various settings (\cite{Kohno}, \cite{OT}). 

\begin{thm}\label{thmKohno}
Suppose $(E,\nabla)$ is generic.

Then, putting $\mathcal{L}=\Ker\nabla^{an}$, we have a canonical isomorphism
$$\Homo_p(X^{an},\mathcal{L})\simeq\Homo_p^{lf}(X^{an},\mathcal{L})\text{ for all }p.$$

Furthermore, we have 
$$\Homo_n^{lf}(X^{an}, \mathcal{L})=\bigoplus_{k}\mathbb{C}\Delta_k\otimes\mathcal{L}_{x_k}$$
and
$$\Homo_p(X^{an}, \mathcal{L})=0\;(p\neq n),$$

where $\Delta_k$ are bounded chambres of $Y\cap\mathbb{R}^n$ and $x_k\in\Delta_k$ is a point.
\end{thm}

\begin{proof}
The first isomorphism was proved in \cite{Kohno}. Let us then discuss the second part.

For any positive real numbers $\eta=(\eta_1,\cdots,\eta_N)\in\mathbb{R}_{>0}^N$, consider the function $F_\eta$ defined by 

$$F_\eta(t)=\sum_{j=1}^N\eta_j\log|l_j(t)|.$$

As in \cite{AK} Chapter 4, one can prove that $F_\eta$ is Morse and that the stable manifolds of $F_\eta$ are exactly bounded chambers $\{\Delta_k\}_k$. Hence, if we consider the gradient flow $\Phi_s$ of $F_\eta$, we have that for any $x\in X\setminus\cup_k\Delta_k$ and any positive real number $R>0$, there exists $s_0>0$ such that for all $s\geq s_0$, $|\Phi_s(x)|>R.$ Therefore, for any open neighbourhood $W$ of $A_{N+1}$ in $\mathbb{P}^n,$ we have a surjection 
$$\Homo_p(W\cap X^{an},\mathcal{L}^\vee)\twoheadrightarrow\Homo_p(X\setminus\cup_k\Delta_k,\mathcal{L}^\vee),$$
or equivalently,
$$\Homo_p^{lf}(X\setminus\cup_k\Delta_k,\mathcal{L})\hookrightarrow\Homo_p^{lf}(W\cap X^{an},\mathcal{L}).$$

Let us construct $W$ so that
$$\Homo^p(W\cap X^{an},\mathcal{L})=0$$
for any k. This implies that $\Homo_p^{lf}(W\cap X^{an},\mathcal{L})=0$ by Poincar\'e duality. Denote by $\iota$ the inclusion $\iota: W\cap X^{an}\hookrightarrow X^{an}$. Firstly, we prove $R^p\iota_*(\mathcal{L})=0$ on some neighbourhood $W$ of $A_{N+1}\subset\mathbb{P}^n$ for all $p$. However, this can be proved thanks to Theorem3.3.7. of \cite{OT}. In fact, for any point $x\in A_{N+1}$, if we take a sufficiently small neighbourhood $W_x$, $W_x\cap X^{an}$ is homotopic to a central arrangement, whose Euler characteristic is 0 by Theorem3.3.7 of \cite{OT}. Since $R^p\iota_*(\mathcal{L})=0$ for $p\neq n$, we can conclude the assertion by the Euler-Poincar\'e formula for local systems.

By the spectral sequence of Leray, we have that for some small neighbourhood $W$ of $A_{N+1},$ $\Homo^p(W\cap X^{an},\mathcal{L})=0$ for any p. We have thus proved that 
$$\Homo_p^{lf}(X\setminus\cup_k\Delta_k,\mathcal{L})=0.$$

Now, by the exact sequence of Gysin, we have
$$\cdots\rightarrow\Homo_p^{lf}(\cup_k\Delta_k,\mathcal{L})\rightarrow\Homo_p^{lf}(X,\mathcal{L})\rightarrow\Homo_p^{lf}(X\setminus\cup_k\Delta_k,\mathcal{L})\rightarrow\cdots\;\;(\text{exact}).$$
Summing up we obtain the theorem.
\end{proof}

Now, let us take a real linear polynomial $f(t)$. For any real number $R>0$, $\nabla|_{f^{-1}(R)\cap X}$ is another regular connection. 

\begin{dfn}
The hyperplane arrangement $\{A_j\}_j$ is said to be asymptotically generic with respect to $f$ if for sufficiently large $R>0,$ the induced connection $(E,\nabla)|_{f^{-1}(R)\cap X}$ is generic.  
\end{dfn}
Note that for sufficiently large $R>0$, connections $(E^{an},\nabla^{an})|_{f^{-1}(R)\cap X}$ are all isomorphic so that this definition is well-defined.

Now, we would like to compute the basis of rapid decay homology group associated to the elementary irregular connection
$$\nabla_f=\nabla-df\wedge.$$

\begin{thm}\label{mainthm}
Suppose $(E,\nabla)$ is generic and is also asymptotically generic. Let $\{\Delta_k\}_k$ denote all bounded chambres of $\mathbb{R}^n\cap X$ and let $\{\tilde{\Delta}_l\}_l$ denote all unbounded chambres which intersect with bounded chambers of $f^{-1}(R)\cap X$ for $R>0$ big enough. By abuse of notations, we denote by $\tilde{\Delta}_l$ the intersection of $\tilde{\Delta}_l$ and $\{ \Re f<R\}$. Then, for any $R>0$ big enough, there is an isomorphism

$$\Homo_p^{r.d.}(X^{an}, (\mathcal{O}_X,\nabla_f))\simeq\Homo^{lf}_p(X^{an}\setminus f^{-1}(R),\mathcal{L})$$

for any $p.$ In particular, $$\Homo_p^{r.d.}(X^{an}, (\mathcal{O}_X, \nabla_f))=0\;(p\neq n).$$

Moreover, 
$$\Homo^{lf}_n(X^{an}\setminus f^{-1}(R),\mathcal{L})=\bigoplus_{k}\mathbb{C}\Delta_k\otimes\mathcal{L}_{x_k}\oplus\bigoplus_l\mathbb{C}\tilde{\Delta}_l\otimes\mathcal{L}_{\tilde{x_l}},$$

where $x_k\in\Delta_k$ and $\tilde{x}_l\in\tilde{\Delta}_l$ are points.
\end{thm}

\begin{proof}
By Lemma\ref{lem14} and Theorem\ref{comparison}, we have the following exact sequence:
$$\cdots\rightarrow\Homo_k(f^{-1}(R);\mathcal{L})\rightarrow\Homo_k(X^{an};\mathcal{L})\rightarrow\Homo_k^{r.d.}(X^{an};\mathcal{L})\rightarrow\cdots\;\; \text{(exact).}$$

We obtain the result in view of Theorem\ref{thmKohno} and the fact that $\Homo_k^{lf}(X^{an}\setminus f^{-1}(R),\mathcal{L})=\Homo_k^{lf}(X^{an},f^{-1}(R);\mathcal{L}).$
\end{proof}

Consider now the case when $f$ is a non-degenerate positive definite quadratic $f=\sum_{i=1}^nt_i^2.$ If we take a sufficiently large real number $R>0$, we can consider the ``asymptotic connection'' $(E,\nabla)|_{f^{-1}(R)\cap X}$. Now, we can compactify $f^{-1}(R)$ as $\bar{Y}_R=\{\sum_{i=1}^nt_i^2=t_0^2R\}\subset\mathbb{P}^n.$ Then, divisors $A_j$ induce divisors $\bar{A_j}$ on $\bar{Y}_R.$ We put $Y_R=\bar{Y}_R\setminus\cup_j\bar{A_j}.$ We consider a family $\bar{l}_j^c(t)=ca_{0j}t_0+a_{1j}t_1+\cdots +a_{nj}t_n$ where $0\leq c\leq 1.$ The corresponding hyperplane is denoted by $A_j^c=\{ \bar{l}^c_j(t)=0\}$, where we put $A_{N+1}^0=A_{N+1}.$ We say that the arrangement $\{ A_j^0\}_{j=1}^{N}$ in $\mathbb{A}^n$ is Boolean if any subset $\mathscr{B}\subset\{ A_j^0\}_{j=1}^{N}$ with cardinality $n$ satisfies $\cap\mathscr{B}=\{ 0\}.$

\begin{thm} 
Suppose that  $\{ A_j^0\}_{j=1}^{N}$ in $\mathbb{A}^n$ is Boolean, the divisor $\cup_{j=1}^{N+1}A_j\cap \bar{Y}_R$ in $\bar{Y}_R$ is normal crossing, $\bar{Y}_R$ is transversal to $\cup_{j=1}^{N+1}A^0_j$, and the divisor $\cup_{j=1}^NA^0_j\cap A_{N+1}\cap\bar{Y}_R$ in $A_{N+1}\cap\bar{Y}_R$ is normal crossing. 

If any eigenvalue of $P_j$ is not an integer, then, we have a canonical isomorphism
$$\Homo_p(Y^{an}_R,\mathcal{L})\simeq\Homo_p^{lf}(Y_R^{an},\mathcal{L})\text{ for all }p.$$

Furthermore, we have 
$$\Homo_n^{lf}(Y_R^{an}, \mathcal{L})=\bigoplus_{l}\mathbb{C}\bar{\Delta}_l\otimes\mathcal{L}_{\bar{x_l}}$$
and
$$\Homo_p(Y_R^{an}, \mathcal{L})=0\;(p\neq n).$$

Here, $\bar{\Delta_l}$ are bounded chambres of the real hypersphere arrangement $\{\mathbb{R}^n\cap Y_R\}.$ 
\end{thm}

\begin{proof}
The first isomorphism follows as in \cite{Kohno} (note that we do not need to employ any blow-up because of our assumption). We notice that there is a homotopy equivalence between $Y_R^{an}$ and $\bar{Y}_R\setminus \cup_jA^0_j.$ This can be constructed by Thom-Mather's 1st isotopy lemma.
Namely, equipping $\bar{Y}_R\times\mathbb{A}^1_c$ with a canonical stratification coming from arrangements $\{ \overline{l}^c_j(t)=ca_{0j}t_0+a_{1j}t_1+\cdots a_{nj}t_n=0\}_{j=1}^N\cup \{ t_0=0\},$ one can confirm that the projection $\pi :\bar{Y}_R\times\mathbb{A}_c^1\rightarrow\mathbb{A}^1_c$ is a stratified submersion  at $c=0,1$ since $\bar{Y}_R$ is transversal to $\{ A_j\}_j$ and $\{ A^0_j\}_j$.
For notational simplicity, we denote $\bar{Y}_R\setminus \cup_jA^0_j$ by $Y^{an}.$

As in Theorem3.3, we consider a function
$$F=\sum_{i=1}^n\log|l_j|$$
on $Y^{an}.$ On each connected component of $Y^{an}\cap\mathbb{R}^n,$ $F$ has at least one critical point. Note that the number of bounded chambers is given by 
$$
M=
\left(
\begin{array}{c}
      N-1 \\
      n-1
    \end{array}
\right)
+\sum_{i=1}^n
\left(
\begin{array}{c}
      N \\
      n-i 
    \end{array}
\right).
$$

By Theorem4.4, Proposition5.2, and Corollary 5.13 of \cite{OTMil}, we can confirm that $F$ has at most M critical points on $Y^{an}.$ Thus, $F$ has no critical points on $Y^{an}\setminus \cup_l\bar{\Delta_l}$. We can now repeat the argument of Theorem 3.3.

\end{proof}

\begin{rem}
Suppose $n=2.$ Then, by a direct computation, one can confirm that the residue matrix at infinity of $\{ t_1^2+t_2^2=Rt_0^2\}$ is given by $P_{N+1}$. The general case is reduced to $n=2$ case by the procedure of slicing.
\end{rem}

We conclude this manuscript with a theorem which concerns the construction of solution basis of irregular Schl\"ofli type.

\begin{thm}\label{mainthm2}
Suppose the same assumption as Theorem 3.5 is satisfied for sufficiently large $R>0$. Then, we have an  isomorphism
$$\Homo_p^{r.d.}(X^{an},(\mathcal{O}_X,\nabla_f))\simeq\Homo_p^{lf}(X^{an}\setminus f^{-1}(R),\mathcal{L})$$
for any $p.$ In particular, 
$$\Homo_p^{r.d.}(X^{an},(\mathcal{O}_X,\nabla_f))=0 \;\; (p\neq n).$$
Furthermore, we have
$$\Homo_n^{lf}(X^{an}\setminus f^{-1}(R),\mathcal{L})=\bigoplus_k\Delta_k\otimes\mathcal{L}_{x_k}\oplus\bigoplus_l\tilde{\Delta}_l\otimes\mathcal{L}_{\tilde{x}_l},$$
where $\Delta_k$ are bounded chambres of arrangements $\{ A_j\}_j$ and $\tilde{\Delta_l}$ are the intersections of $\{ \Re(f)<R\}$ and unbounded chambres of the same arrangements.
\end{thm}

We will conclude this paper with a more concrete description of the basis of rapid decay homologies under a certain situation. Let us take a closer look at Theorem \ref{thmKohno}. The isomorphism $$\Homo_p(X^{an},\mathcal{L})\tilde{\rightarrow}\Homo_p^{lf}(X^{an},\mathcal{L})\text{ for all }p$$ is naturally induced by the definition of locally finite homology. On the other hand, its inverse is, in general, difficult to describe. However, there is a concrete description of the inverse which is called ``regularization'' and was developed by K. Aomoto when the rank of the local system is 1. Let me briefly describe the method following \cite{G}. (see also Chapter 3, section 2 of \cite{AK}).

Hereafter, we assume that $\rank\mathcal{L}=1,$ and that the arrangement $\{ A_j\}_{j=1}^N$ is in general position in the sense of \cite{AK}, i. e., the arrangement $\{ a_{0j}t_0+a_{1j}t_1+\cdots +a_{nj}t_n=0\}_{j=0}^N$ in $\mathbb{A}^{n+1}$ is Boolean. Let $\Delta_k$ be a bounded chambre of the arrangement $\{ A_j\}_{j=1}^N.$ For simplicity, let us assume that $\Delta_k$ is surrounded by $A_1, \cdots, A_t\; (t\leq N)$. For any subset $J\subset\{ 1,\cdots ,t\},$ we put 
$$A_J=\cap_{j\in J}A_j,\;\Delta_J=\bar{\Delta}_k\cap A_J,\; T_J=\varepsilon\text{-neighbourhood of }\Delta_J.$$
($\varepsilon $ is supposed to be small enough.)  We also put
$$\sigma_k=\Delta_k\setminus\cup_{J}T_J$$
and equip with it a standard orientation coming from that of $\mathbb{R}^n.$
For any $j\in J,$ let $l_j$ be the $n-1$ face of $\sigma_k$ defined by $\sigma_k\cap\bar{T}_j.$ Let $S_j$ be a circle going around $A_j$ in a positive direction and whose starting point is $l_j$. Suppose the monodromy of $\mathcal{L}$ around $A_j$ is $\exp(2\pi\sqrt{-1}\alpha_j)\neq 1$ and we put $d_j=\exp(2\pi\sqrt{-1}\alpha_j)-1.$ Then, we put
$$\Delta_k^{reg}=\sigma_k+\displaystyle\sum_{\phi\neq J\subset\{ 1,\cdots ,t\}}\prod_{j\in J}\frac{1}{d_j}\Bigl( \bigl( \bigcap_{j\in J}l_j\bigr)\times \prod_{j\in J}S_j\Bigr).$$
Precisely speaking, we have to be more careful about its orientation. See \cite{AK}. 
It can readily be seen that $$[\Delta^{reg}_k]=[\Delta_k]\text{ in }\Homo_n^{lf}(X^{an},\mathcal{L})$$ where the bracket stands for the homology class in $\Homo_n^{lf}(X^{an},\mathcal{L}).$

This process is called the regularisation process. Regularization can be performed even for unbounded chambres $\tilde{\Delta}_l$ of Theorem \ref{mainthm}. It can be seen that the regularised cycle $\tilde{\Delta}^{reg}_l$ belongs to $\tilde{\Homo}^{r.d.}_n(X^{an},\mathcal{L})\simeq\Homo_n^{r.d.}(X^{an},\mathcal{L}).$ We obtain the following

\begin{crl}\label{cor}
Under the same assumption as Theorem \ref{mainthm}, if further $\{A_j\}_j$ is in general position and if $\rank\mathcal{L}=1$, we have the following decomposition of the rapid decay homology group.
$$\Homo_n^{r.d.}(X^{an}, \mathcal{L})=\bigoplus_{k}\mathbb{C}\Delta_k^{reg}\oplus\bigoplus_l\mathbb{C}\tilde{\Delta}_l^{reg}.$$ 
\end{crl}

\begin{rem}
We can obtain a similar result under the assumption of Theorem \ref{mainthm2}. The author hopes that Corollary \ref{cor} and its variant for Theorem \ref{mainthm2} will be the starting point of the global analysis of various special functions (see \cite{AK} and \cite{G}).

\end{rem}



\end{document}